\documentclass[12pt,a4paper]{article}
\usepackage{german,theorem,latexsym,amsfonts}
\usepackage{amssymb,amsmath}
\usepackage{epsfig}
\usepackage{color}
\usepackage[all]{xy}
\theoremstyle{change}
\theorembodyfont{\rm}
\newtheorem{prop}{Proposition:}[section]
\newtheorem{coro}[prop]{Corollary:}
\newtheorem{defi}[prop]{Definition:}
\newtheorem{exa}[prop]{Example:}
\newtheorem{exa_n}[prop]{Examples:}
\newtheorem{nr}[prop]{}

\newtheorem{theo}[prop]{Theorem:}
\newenvironment{rema}{\par\noindent\textbf{Remark:}}{}
\newenvironment{proof}{\par\noindent\textbf{Proof:}}{\hfill\ensuremath{\square}}

\def\colim{\mathop{\rm colim}}
\def\hocolim{\mathop{\rm hocolim}}

\def\id{\mathop{\rm id}}
\def\Supp{\mathop{\rm Supp}}
\def\mC{\mathbb{C}}
\def\Top{\mathop{\rm Top}}
\def\Mon{\mathop{\rm Mon}}
\def\op{\mathop{\rm op}}
\def\M{\widehat{M}}

\begin{document}
\title{\textbf{Numerably Contractible Spaces}}
\author{\textbf{Eugenia Schwamberger and Rainer Vogt}}
\date{\textbf{24. 10. 2008}}
\maketitle
%\thispagestyle{empty}
%\newpage

%\tableofcontents
%\newpage

%\sec1
\begin{abstract}
Numerably contractible spaces play an important role in the theory of homotopy
pushouts and pullbacks. The corresponding results imply that
a number of well known weak homotopy equivalences are
genuine ones if numerably contractible spaces are involved. 
In this paper we give a first systematic investigation of numerably contractible spaces.
We list the elementary properties of the category of these spaces.
We then study simplicial objects in this category. In particular, we show that the topological 
realization functor preseves fibration sequences if the base is path-connected 
and numerably contractible
in each dimension. Consequently, the loop space functor commutes with
realization up to homotopy. Finally, we give simple conditions which assure that
free algebras over a topological operad are numerably contractible.
\end{abstract}
\section{Introduction}
A \textit{numerably contractible space} is a topological space $X$ which admits a
numerable cover by sets $U\subset X$ for which the inclusions are nullhomotopic.
Numerably contractible spaces are of importance in homotopy theory, as we will
explain.

Some important weak homotopy equivalence are strict ones if the spaces involved are
numerably contractible. Let $k\Top^\ast$ denote the category of based $k$-spaces. For
$X$ in $k\Top^\ast$ let $JX$ denote the James construction on $X$ in $k\Top^\ast$,
i.e. the based free topological monoid on $X$. In \cite[(17.3)]{DKP} and \cite[Cor.
3.4]{Pup1} D. Puppe proved:

\begin{theo}\label{1_1}
If $X$ is $h$-wellpointed, path-connected and numerably contractible, then
$JX\simeq\Omega\Sigma X$.
\end{theo}

For $X$ in $k\Top^\ast$ let $\mC^\ast_n(X)$ denote the based free algebra over the operad
$\mathcal{C}_n$ of little $n$-cubes. P. May constructed a
weak equivalence $\mC^\ast_n(X)\to\Omega^n\Sigma^nX$ for a path-connected $X$ \cite{May}. 
In his thesis H. Meiwes
proved \cite{Meiwes}.

\begin{theo}\label{1_2}
If $X$ is as in Theorem \ref{1_1}, then May's map
$\mC^\ast_n(X)\to\Omega^n\Sigma^nX$ is a genuine homotopy equivalence.
\end{theo}

In the context of these theorems we share D. Puppe's point of view \cite{Pup1}: 
``Frequently a weak
homotopy equivalence is considered as good as a genuine one, because for spaces having
the homotopy type of a $CW$-complex there is no difference and most interesting 
spaces in algebraic topology are of that kind. I am not going to argue against this 
because I agree with it, but I do think that the methods by which we establish
the genuine homotopy equivalences give some new insight into homotopy theory.''

Indeed, constructing homotopy equivalences between spaces which are not
necessarily of the homotopy type of $CW$-complexes deprives one of the algebraic side
of homotopy theory, so that these constructions have a different, more geometric
flavor.

We do not know whether A. Dold introduced the notion of a numerably contractible
space, but he was certainly among the first ones to work with them. Following J.
Smrekar \cite{Smrekar}, we therefore also call such a space a Dold space. In his
paper ``Partitions of Unity in the Theory of Fibrations'' Dold proved
\cite[Thm. 6.3]{Dold}:

\begin{theo}\label{1_3}
Given a commutative diagram
$$
\xymatrix{
E\ar[rr]^f\ar[rd]_p && E' \ar[ld]^{p'}
\\
& B
}
$$
such that $p$ and $p'$ have the weak covering homotopy property and $B$ is a Dold space
then $f$ is a fiberwise homotopy equivalence iff its restriction to
every fiber is a homotopy equivalence.
\end{theo}

As simple consequences of this result one has the following strengthened versions of
well-known results about homotopy pullbacks (for simplicity we state the results for
commutative squares. They also hold for homotopy commutative squares with a
specified homotopy.)

\begin{prop}\label{1_4}
Let
$$
\xymatrix{
X_1 \ar[r]^u\ar[d]_f & Y_1 \ar[d]^g
\\
X_0 \ar[r]^v & Y_0
}
$$
be a homotopy pullback. If $v$ is a homotopy equivalence, so is $u$.
Conversely, if $u$ is a homotopy equivalence, $g$ induces a surjection of
sets of path-components, and $Y_0$ is a Dold space, then $v$ is a homotopy
equivalence.
\end{prop}

\begin{prop}\label{1_5}
Given a commutative diagram,
$$
\xymatrix{\
X_2 \ar[r]^{f'} \ar[d]^w \ar @{} [dr] |{\textrm{I}}&
X_1 \ar[r]^f \ar[d]^v \ar @{} [dr] |{\textrm{II}} &
X_0 \ar[d]^u
\\
Y_2 \ar[r]^{g'} &
Y_1 \ar[r]^g &
Y_0
}
$$
\begin{enumerate}
\item suppose that II is a homotopy pullback. Then I is a homotopy pullback iff the
combined square I+II is a homotopy pullback
\item suppose that I and I+II are homotopy pullbacks, that $g'$ induces a surjection
of sets of path-components and $Y_1$ is a Dold space, then II is a
homotopy pullback.
\end{enumerate}
\end{prop}

\begin{prop}\label{1_6}
Let
$$
\xymatrix{
X_1 \ar[r]^f\ar[d]_u & X_0 \ar[d]^v
\\
Y_1 \ar[r]_g & Y_0
}
$$
be a commutative square and $F(f,x)$ the homotopy fiber of $f$ over $x\in X_0$.
\begin{enumerate}
\item If the square is a homotopy pullback, then the induced map 
$$
F(f,x)\to F(g,v(x))
$$
is a homotopy equivalence for each $x\in X_0$.
\item If for each $x\in X_0$ the map
 $F(f,x)\to F(g,v(x))$ is a homotopy equivalence and
$X_0$ is a Dold space, the square is a homotopy pullback.
\end{enumerate}
\end{prop}

We also have the following improved version of M. Mather's second cube theorem
\cite{Mather}.

\begin{prop}\label{1_7}
Given a commutative cube diagram whose vertical faces are homotopy pullback,
$$
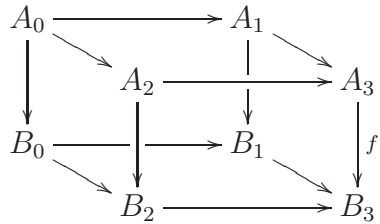
\xymatrix@=1.5ex{
A_0 \ar[rrrr]\ar[dd] \ar[drr] &&&&
A_1  \ar[dd]|!{[d];[d]}\hole \ar[drr]
\\
&&
A_2 \ar[rrrr]\ar[dd] &&&&
A_3 \ar[dd]^f
\\
B_0 \ar[rrrr]|!{[rr];[rr]}\hole \ar[drr]  &&&&
B_1 \ar[drr] 
\\
&&
B_2 \ar[rrrr] &&&&
B_3
}
$$
then
\begin{enumerate}
\item the top face is a homotopy pushout if the buttom face is a homotopy pushout.
\item the bottom face is a homotopy pushout if the top face is a homotopy pushout,
$f$ induces a surjection on path-components, and $B_3$ is a Dold space.
\end{enumerate}
\end{prop}

Homotopy pushouts and pullbacks have become increasingly important tools in
homotopy theory and homological algebra. E.g. there
exist comparatively simple proofs of Theorem \ref{1_1} based on
Propositions \ref{1_4} to \ref{1_6} (unpublished). 
Another example is the following result of G. Allaud \cite{Allaud}, which
is an immediate consequece of Propositions \ref{1_4} and \ref{1_5}.

\begin{prop}\label{1_8}
Let $f:X\to Y$ be a based map of path-connected Dold spaces such
that $\Omega f: \Omega X\to\Omega Y$ is a homotopy equivalence. Then $f$ is a
homotopy equivalence.
\end{prop}

So we feel that it is time for
a more systematic investigation of Dold spaces.

In Section 2 we will recall the definition of Dold spaces and some facts about 
numerable covers. In Section 3 we will list a number of elementary facts about
Dold spaces. Section 4 is the main part of the paper: We will study properties of
simplicial Dold spaces and their realizations. We give a characterization of
wellpointed connected Dold spaces and use it to derive results about the 
realization of maps of simplicial spaces which are dimensionwise fibrations.
In Section 5 we apply these results to free algebras over topological
operads. We close the paper with a section on counter examples.

Some of our results are well-known, some are known to specialists but have
not appeared in print, some are new. We derive most of the well-known facts
as special cases of more general results. We have tried to give references as
well as possible, but we are not sure that we always found the original source. 

We are indepted to A. Hatcher for bringing Example \ref{6_1} to our attention
and to J. Smrekar for suggesting the name ``Dold space'' and for
e-mail exchange about function space properties of Dold spaces. The latter
turned out to be so retrictive that we did not include them in this paper.
In fact, the category of Dold spaces is rather badly behaved with respect
to function spaces. In particular, loop spaces of Dold spaces need not be 
Dold spaces (see Example \ref{6_3}).

%\sec2
\section{Dold covers}
In this section we recall the basic definitions and list results related to
coverings.

Let $\{a_j; j\in J\}$ denote a set of elements of $\mathbb{R}_+=\{x\in\mathbb{R};
x\ge 0\}$. We define
$$
\sum\limits_{j\in J} a_j=\sup\left\{\sum\limits_{j\in E} a_j;\ \ E \subset J \textrm{
finite}\right\}
$$

\begin{defi}\label{2_1}
A \textit{partition of unity} on a space $X$ is a set of maps $\{f_j:X\to[0,1]; j\in
J\}$ such that
$$
\sum\limits_{j\in J}f_j(x)=1\quad \textrm{ for all } x\in X.
$$
\end{defi}

\begin{defi}\label{2_2}
Let $X$ be a space. A subset $A\subset X$ is called \textit{ambiently contractible}
if the inclusion $A\to X$ is nullhomotopic.
\end{defi}

\begin{defi}\label{2_3}
Let $\mathcal{U}=\{U_\alpha; \alpha\in A\}$ be a cover of $X$.
\begin{enumerate}
\item $\mathcal{U}$ is called \textit{locally finite} if each $x\in X$ has a
neighborhood $V$ such that $V\cap U_\alpha\neq\emptyset$ for only finitely many
$\alpha\in A$.
\item A \textit{numeration of} $\mathcal{U}$ is a partition of unity $\{f_\alpha;
\alpha\in A\}$ of $X$ such that $\{\Supp (f_\alpha); \alpha\in A\}$ is locally
finite and $\Supp (f_\alpha)\subset U_\alpha$ for all $\alpha\in A$. (Recall, the
support $\Supp(f)$ of a map $:X\to I$ is the closure of $f^{-1}(]0,1])$.)
If $\mathcal{U}$ admits a numeration, it is called a \textit{numerable cover}.
\item $\mathcal{U}$ is called an \textit{ambiently contractible cover} if each
$U_\alpha$ is ambiently contractible.
\item $\mathcal{U}$ is called a \textit{Dold cover} if it is numerable and
ambiently contractible.
\end{enumerate}
\end{defi}

\begin{defi}\label{2_4}
A space $X$ is called \textit{ambiently locally contractible} if is has a ambiently
contractible open cover. We call $X$ \textit{numerably contractible} or \textit{Dold
space} if it has a Dold cover.
\end{defi}
\begin{rema}\label{2_4a}
In the literature the term ``weakly contractible'' is used for what we
call ``ambiently contractible''. 
\end{rema}
\begin{exa_n}\label{2_5}
\begin{enumerate}
\item Each discrete space is a Dold space.
\item Each contractible space is a Dold space.
\item Each paracompact ambiently locally contractible space is a Dold space.
\item By \cite[Thm. II. 3]{DE} each paracompact $LEC$ space is a Dold space.
\item Each $CW$-complex is a Dold space (see \ref{3_8}). 
The converse does not hold (see \ref{6_1}).
\end{enumerate}
\end{exa_n}

We will make use of the following results.

\begin{nr}\label{2_6} \textbf{Lemma} \cite[p. 347]{tD}:
If $\{f_j,j\in J\}$ is a partition of unity on $X$, then $\{f^{-1}_j(]0,1]); j\in
j\}$ is a numerable cover of $X$.
\end{nr}

\begin{nr}\label{2_7} \textbf{Lemma} \cite[p. 349]{tD}:
Let $\{f_j:X\to\mathbb{R}_+; j\in J\}$ be a set of maps such that
$\mathcal{U}=\{f^{-1}(]0,\infty[); j\in J\}$ is a locally finite cover of $X$, then
$\mathcal{U}$ is a numerable cover.
\end{nr}

%\pagebreak
\begin{coro}\label{2_8}
The following are equivalent
\begin{enumerate}
\item $X$ is a Dold space
\item $X$ has a partition of unity $\{f_j; j\in J\}$ such that $\{f^{-1}_j(]0,1]);
j\in J\}$ is an ambiently contractible open cover of $X$.
\item there is a set of maps $\{f_j:X\to\mathbb{R}_+; j\in J\}$ such that
$\{f^{-1}_j(]0,\infty[); j\in J\}$ is a locally finite ambiently contractible cover of
$X$.
\end{enumerate}
\end{coro}

\begin{prop}\label{2_9}
Let $\mathcal{U}=\{U_\alpha; \alpha\in A\}$ be a cover of $X$ by Dold spaces.
Suppose that $\mathcal{U}$ has a numerable refinement $\mathcal{V}=\{V_j; j\in J\}$,
i.e. a numerable cover of $X$ such that each $V_j$ is contained in some $U_\alpha$.
Then $X$ is a Dold space.
\end{prop}

\begin{proof}
Let $\{f_j; j\in J\}$ be a numeration of $\mathcal{V}$. Since each $U_\alpha$ is a
Dold space there are partitions of unity
$$
\{g_{\alpha,k}:U_\alpha\to I; k\in K_\alpha\}
$$
such that $\{\Supp(g_{\alpha,k}); k\in K_\alpha\}$ is locally finite and
$\Supp(g_{\alpha,k})$ is contractible in $U_\alpha$ and hence in $X$ for all $k\in
K_\alpha$.

Choose a function $\beta:J\to A$ such that $V_j \subset U_{\beta(j)}$.
For $k\in K_{\beta(j)}$ define $f_{j,k}: X\to [0,1]$ by
$$
f_{j,k}(x)=\left\{
\begin{array}{ll}
f_j(x)\cdot g_{\beta(j),k}(x) \quad & \textrm{ for } x\in\Supp(f_j)\\
0 & \textrm{ for } x\in X\backslash f^{-1}_j(]0,1])
\end{array}
\right.
$$
Then $f_{j,k}$ is well-defined and continuous, because $\Supp(f_j)$ and $X\backslash
f^{-1}_j(]0,1])$ are closed in $X$. The collection $\{f_{j,k}; j\in J, k\in
K_{\beta(j)}\}$ is a partition of unity and
$$
f^{-1}_{j,k}(]0,1])=f^{-1}_j(]0,1])\cap g_{\beta(j),k}(]0,1])
\subset\Supp(g_{\beta(j),k}).
$$
Hence $f^{-1}_{j,k}(]0,1])$ is contractible in $X$. Now apply \ref{2_8}(2).
\end{proof}
\vspace{2ex}

\begin{rema}
The numeration condition on the cover $\mathcal{U}$ in \ref{2_9} is essential as the
following example shows.
\end{rema}

\begin{exa_n}\label{2_10}
Let $X\subset\mathbb{R}^2$ be the cone on $M=\{(0,0)\}\cup\{(\frac{1}{n},0);
n\in\mathbb{N}\backslash\{0\}\}$ with cone point $(0,1)$. Then $X$ is a Dold space.
Now let
$$
Y=(X\sqcup X)/(0,0)\sim (0,0).
$$
The two copies of $X$ form a closed cover of $Y$, but $Y$ is not a Dold space,
because no open neighborhood of $(0,0)$ is contractible in $Y$.
\end{exa_n}

%\sec3
\section{Elementary properties}
Suppose $U\subset X$ is ambiently contractible in $X$ to a point $x_0$, then $U$
must lie in the path-component of $x_0$. We obtain

\begin{prop}\label{3_1}
(1) $X$ is a Dold space iff its path-components are open and Dold spaces.\\
(2) If $X=\coprod_{j\in J} X_j$, then $X$ is a Dold space iff each summand
$X_j$ is a Dold space.
\end{prop}

This observation allows us to restrict our attention to path-connected Dold spaces.

\begin{prop}\label{3_2}
A space $Y$ dominated by a Dold space $X$ is itself a Dold space.
\end{prop}

\begin{proof} (\cite[p. 235]{DKP}
Let $\{V_\lambda; \lambda\in\Lambda\}$ be a Dold cover of $X$, and $f:X\to Y$ and
$g:Y\to X$ be maps such that $f\circ g\simeq\id_Y$. Then $\{g^{-1}(V_\lambda);
\lambda\in\Lambda\}$ is a numerable cover of $Y$ and each $g^{-1}(V_\lambda)$ is
contractible in $Y$ because
$$
\xymatrix{
g^{-1}(V_\lambda )\ar[r]^g & V_\lambda\subset X \ar[r]^(0.6)f & Y
}
$$
is nullhomotopic and homotopic to the inclusion $g^{-1}(V_\lambda)\subset Y$.
\end{proof}

\begin{coro}\label{3_3}
If $X$ and $Y$ are homotopy equivalent then $X$ is a Dold space iff $Y$ is a Dold
space.
\end{coro}

\begin{prop}\label{3_4}
Given a diagram
$$
\xymatrix{
X & A \ar[l]_f\ar[r]^g & Y
}
$$
with $X$ and $Y$ Dold spaces, then the double mapping cylinder $\M(f,g)$ is a Dold
space.
\end{prop}

\begin{proof}
$\M(f,g)=(X\sqcup A\times[0,1]\sqcup Y)/\sim$ with $(a,0)\sim f(a)$ and $(a,1)\sim
g(a)$.

Define $\alpha:\M(f,g)\to [0,1]$ by
$$
\alpha(z)=\left\{
\begin{array}{ll}
0 & z\in X\\
t & z=(a,t)\in A\times[0,1]\\
1 & z\in Y
\end{array}\right.
$$
and $\beta: \M(f,g)\to [0,1]$ by $\beta(z)=1-\alpha(z)$. Then $\{\alpha,\beta\}$ is a
numerable cover, $\alpha^{-1}(]0,1])\simeq Y$ and $\beta^{-1}(]0,1])\simeq X$. Hence
$\M(f,g)$ is a Dold space by \ref{2_9}.
\end{proof}

Recall that a map $f:A\to X$ is an \textit{$h$-cofibration} if there is a
commutative triangle
$$
\xymatrix{
& A\ar[ld]_f\ar[rd]^j\\
X \ar[rr]^h && Y
}
$$
with $j$ a cofibration and $h$ a homotopy equivalence under $A$. Dually, an
\textit{$h$-fibration} is a map $f:A\to X$ which is homotopy equivalent
over $X$ to a fibration.

\begin{coro}\label{3_5}
Let 
$$
\xymatrix{
A\ar[r]^f\ar[d]_g & B\ar[d]\\
C\ar[r] & X
}
$$
be a pushout square with $f$ an $h$-cofibration and $B$ and $C$ Dold spaces. Then
$X$ is a Dold space.
\end{coro}

\begin{proof}
Since $f$ is an $h$-cofibration, the canonical map $\widehat{M}(f,g)\to X$ is a homotopy
equivalence.
\end{proof}

\begin{coro}\label{3_6}
\begin{enumerate}
\item The unreduced suspension $\widehat{\Sigma} X$ 
of any space $X$ is a Dold space.
\item \cite[Lemma 1.3]{Pup1} For any map $f:A\to X$ into a Dold space $X$, 
the unreduced mapping cone is a Dold space.
\item If $f:A\to X$ is an $h$-cofibration and $X$ a Dold space, then $X/f(A)$ is a Dold space.
\end{enumerate}
\end{coro}

\begin{prop}\label{3_7} \cite[Lemma 1.6]{Pup1}
Let $\xymatrix{X_0\ar[r]^{f_0}& X_1\ar[r]^{f_1}& X_2\ar[r]^{f_2}& \ldots}$ be a
sequence of maps of Dold spaces. Then
\begin{enumerate}
\item the mapping telescope $TX=(\coprod_{n\ge 0} X_n\times I)/\sim$ with $(x,1)\in
X_n\times I$ related to $(f_n(x),0)\in X_{n+1}\times I$ is a Dold space.
\item if  each $f_i$ is an h-cofibration, $\colim X_n$ is  a Dold space.
\end{enumerate}
\end{prop}

\begin{proof}
$TX$ is the double mapping cylinder of 
$$
\xymatrix{
\coprod\limits_{n\textrm{ even}} X_n
& \coprod\limits_{n\le 0}X_n \ar[l]_(0.4){(g_n)} \ar[r]^(0.4){(h_n)}
& \coprod\limits_{n\textrm{ odd}} X_n
}
$$
with $g_n(x)=\left\{
\begin{array}{ll}
x & n\textrm{ even}\\
f_n(x) & n\textrm{ odd}
\end{array}\right.
\qquad
h_n(x)
=\left\{
\begin{array}{ll}
f_n(x) & n\textrm{ even}\\
x & n\textrm{ odd}
\end{array}\right.
$

If all the $f_n$ are h-cofibrations, the canonical maps $TX\to \colim X_n$ is homotopy
equivalence.
\end{proof}

\begin{coro}\label{3_8} \cite[Prop. 6.7]{Dold}
Each $CW$-complex $X$ and hence each space of the homotopy type of $CW$-complex is a
Dold space.
\end{coro}

\begin{proof}
Let $X^{(n)}$ denote the $n$-skeleton of $X$. Then $X^{(n)}$ is a Dold space by
induction on $n$ using \ref{3_5}. Hence $X$ is a Dold space by \ref{3_7}.2.
\end{proof}

\begin{prop}\label{3_9}
Let $p:E\to B$ be any map. Assume that $B$ and the homotopy fibers 
$F_b(p)$ of $p$ over $b$ are Dold spaces for all $b\in B$, 
then $E$ is a Dold space.
\end{prop}

\begin{proof}
By \ref{3_1} we may assume that $B$ is path-connected, and
by \ref{3_3} we may assume that $p:E\to B$ is a fibration whose fiber $F$
over a fixed $b_0\in B$ is a Dold space.
Let $\mathcal{U}=\{U_\lambda; \lambda\in\Lambda\}$ be an open Dold cover 
of $B$ and $\{f_\lambda: B\to[0,1];\lambda\in\Lambda\}$ a
numeration of $\mathcal{U}$, and let $\mathcal{V}=\{V_\gamma, \gamma\in\Gamma\}$
with $\{g_\gamma: F\to[0,1];\gamma\in\Gamma\}$ be 
the corresponding data for $F$. Let
$$
H_\lambda:U_\lambda\times I\to B
$$
be a homotopy from the inclusion $i_\lambda:U_\lambda\subset B$ to the constant map
to $b_0$. Since $p$ is a fibration there is a homotopy $K_\lambda$
$$
\xymatrix@R=1ex{
p^{-1}(U_\lambda)\times 0\; \ar@{^{(}->}[rr] 
&& E \ar[dd]^p
\\
\bigcap
\\
p^{-1}(U_\lambda)\times I \ar[rr]^(0.6){H_\lambda\circ(p\times \id)} \ar[uurr]^{K_\lambda}
&& B
}
$$
from the inclusion $j_\lambda:p^{-1}(U_\lambda)\subset E$ to a map
$p^{-1}(U_\lambda)\stackrel{k_\lambda}{\to}F\subset E$. Define maps
$\tau_{\lambda,\gamma}:E\to[0,1]$ by
$$
\tau_{\lambda,\gamma}(e)=\left\{
\begin{array}{ll}
f_\lambda(p(e))\cdot g_\gamma(k_\lambda(e))
& e\in p^{-1}(\Supp(f_\lambda))\\
0 & e\notin p^{-1}(f^{-1}_\lambda(]0,1])
\end{array}\right. .
$$
Since the $k^{-1}_\lambda(g^{-1}_\gamma(]0,1]))$, $\gamma\in\Gamma$, cover
$p^{-1}(U_\lambda)$ and since the $p^{-1}(f^{-1}_\lambda(]0,1]))
\linebreak
\subset p^{-1}(U_\lambda)$, $\lambda\in\Lambda$, cover $E$,
$$
\mathcal{W}=\{W_{\lambda,\gamma}=\tau^{-1}_{\lambda,\gamma}(]0,1]);
\lambda\in\Lambda,\gamma\in\Gamma\}
$$
covers $E$. This cover is locally finite and ambiently contractible: $K_\lambda$
deforms $W_{\lambda,\gamma}$ into
$k_\lambda(W_{\lambda,\gamma})\subset V_\gamma\subset F$, and $V_\gamma$ is ambiently
contractible in $F$.
Let $e\in E$. Then there is an open neighborhood $U$ of $p(e)$ such that $U\cap
U_\lambda=\emptyset$ for all but finitely many $\lambda_1,\ldots,\lambda_n$. So
$p^{-1}(U)$ only meets $p^{-1}(U_{\lambda_i})$, $i=1,\ldots,n$. Each
$k_{\lambda_i}(e)$ has an open neighborhood $V_i$ such that $V_i\cap
V_\gamma=\emptyset$ for all but finitely many $\gamma_{i1},\ldots,\gamma_{ir_i}$.
Then $p^{-1}(U)\cap\bigcap\limits^n_{i=1}k^{-1}_{\lambda_i}(V_i)\cap
\tau^{-1}_{\lambda,\gamma}(]0,1])\neq\emptyset$ only if
$(\lambda,\gamma)\in\{(\lambda_i,\gamma_{ij}); 1\le i\le n, 1\le j\le r_i\}$.

So $E$ is a Dold space by \ref{2_8}.
\end{proof}

\begin{coro}\label{3_10} \cite[Lemma 1.5]{Pup1}
If $X$ and $Y$ are Dold spaces so is $X\times Y$.\hfill$\square$
\end{coro}

\begin{coro}\label{4_10a}
Let
$$
\xymatrix{
P \ar[rr]^g \ar[d]^q && E\ar[d]^p\\
X \ar[rr]^f && B
}
$$
be a homotopy pullback, let $X$ and the homotopy fibers
$F_b(p)$ of $p$ over all $b\in B$ be Dold spaces. Then 
$P$ is a Dold space.
\end{coro}

\begin{proof}
The homotopy fiber $F_x(q)$ of $q$ over $x\in X$ is homotopy equivalent
to $F_{f(x)}(p)$ and hence a Dold space.
\end{proof}

\begin{coro}\label{3_11}
Given a diagram
$$
\xymatrix@R=1ex{
P \ar[rrr] \ar[ddd]
&&&  X \ar[ddd]\ar[ddl]\\
\\
&& Q \\
Y \ar[rrr]\ar[rru]
&&& B
}
$$
with $B$ a path-connected Dold space, whose outer square is a homotopy pullback and
whose inner square is a homotopy pushout, then $Q$ is a Dold space.
\end{coro}

\begin{proof}
If $F(f)$ and $F(g)$ are the homotopy fibers of $f$ and $g$ respectively, the
homotopy fiber of the induced map $r:Q\to B$ is homotopy
equivalent to the join $F(f)\ast F(g)$ (e.g. see \cite[Prop. 5.5]{Vogt}). Since the
join of two spaces is Dold space, the result follows.
\end{proof}

%\sec4
\section{Simplicial spaces}
Let $\bigtriangleup$ denote the category of finite ordered sets
$[n]=\{0<1<\ldots<n\}$ and order preserving maps and $\Mon\bigtriangleup$ the
subcategory of injective order preserving maps. A \textit{simplicial space} is a
functor $X_\bullet:\bigtriangleup^{\op}\to\Top$, $[n]\mapsto X_n$, a
\textit{semisimplicial space} is a functor $X_\bullet:
(\Mon\bigtriangleup)^{\op}\to\Top$.

Let $|X_\bullet|$ denote the usual topological realization, also called \textit{slim
realization} of the simplicial space $X_\bullet$ and $||X_\bullet||$ denote the
realization of the semisimplicial space $X_\bullet$, also called
\textit{fat realization}. Since a simplicial space can be considered as a 
semisimplicial
one it has a fat and a slim realization. A simplicial space
$X_\bullet$ is called \textit{proper} if the inclusions  $sX_n\subset X_n$ of the
subspaces of degenerate elements of $X_n$ are cofibrations for all~$n$.

\begin{prop}\label{4_1}
\begin{enumerate}
\item If $X_\bullet$ is a semisimplicial space such that $X_0$ is a Dold space, then
$||X_\bullet||$ is a Dold space.
\item If $X_\bullet$ is a proper simplicial space such that $X_0$ is a Dold space, then
$|X_\bullet|$ is a Dold space.
\end{enumerate}
\end{prop}

\begin{proof}
(1) (The idea of the proof is probably due to D. Puppe. We learnt it many
years ago from H. Meiwes \cite{Meiwes2}.)

Let $||X||^{(n)}$ denote the $n$-skeleton of the fat realization. Since
$||X||^{(n)}\subset ||X||^{(n+1)}$ is a cofibration it suffices to show that each
$||X||^{(n)}$ is a Dold space. Assume inductively that $||X||^{(n-1)}$ is a Dold
space. Recall that
$$
||X||^{(n)}=||X||^{(n-1)}\cup_{X_n\times\partial\bigtriangleup^n}
X_n\times\bigtriangleup^n,
$$
where $\bigtriangleup^n$ is the standard $n$-simplex.
Choose two different points $u_1\neq u_2$ in the interior of $\bigtriangleup^n$. For
a space $Y$ let $CY=(Y\times I)/(Y\times 0)$ be the cone on $Y$ with cone-point
$\ast$, and $\varphi:CY\to I$ the map $(y,t)\mapsto t$. Define maps
$$
\lambda_i:(\bigtriangleup^n,u_i)\stackrel{h_i}{\to}(C(\partial\bigtriangleup^n),
\ast) \stackrel{\varphi}{\to}(I,0) \qquad i=1,2
$$
by choosing based homeomorphisms $h_i$. The maps
$$
X_n\times\bigtriangleup^n \stackrel{\textrm{proj.}}{\to}\bigtriangleup^n
\stackrel{\lambda_i}{\to} I
$$
together with the constant map to 1 on $||X||^{(n-1)}$ define maps
$$
f_i:||X||^{(n)}\to I.
$$
Then $\{f^{-1}_1(]0,1]),f^{-1}_2(]0,1])\}$ is a numerable cover of
$||X||^{(n)}$ by \ref{2_7}. The subspaces
$$
f^{-1}_i(]0,1])=||X||^{(n-1)}\cup_{X_n\times\partial\bigtriangleup^n}
X_n\times(\bigtriangleup^n\backslash\{u_i\}).
$$
deformation retract onto $||X||^{(n-1)}$. Hence they are Dold spaces. So
$||X||^{(n)}$ is a Dold space by \ref{2_9}.

If $X$ is a proper simplicial space the natural map $||X||\to |X|$ is a homotopy
equivalence. Hence $|X|$ is a Dold space.
\end{proof}

\begin{coro}\label{4_2}
Let $J$ be small category and $D:J\to\Top$ a diagram of Dold spaces. Then $\hocolim
D$ is a Dold space.
\end{coro}

\begin{proof}
$\hocolim D$ is a topological realization of the proper simplicial space
$$
[n]\mapsto \coprod\limits_{i,j\in J} J_n(i,j)\times D(i)
$$
with $J_n(i,j)=\{(\alpha_1,\ldots,\alpha_n)\in(mor J)^n$;
$\alpha_1\circ\ldots\circ\alpha_n:i\to j\}$ for $n>0$ and
$$
J_0(i,j)=\left\{
\begin{array}{ll}
\id & i=j\\
\emptyset & i\neq j
\end{array}\right.
$$
Its $0$-th space is $\coprod\limits_{j\in J}D(j)$ and hence a Dold space.
\end{proof}

We now consider the based case

\begin{defi}\label{4_2a}
We call a based space
$(X,x_0)$ \textit{wellpointed}, if the inclusion $\{x_0\}\subset X$
is a closed cofibration, and \textit{h-wellpointed} if it is an h-cofibration.
\end{defi}

The homotopy colimit of a diagram $D$ generally will have a different homotopy type, when
taken in the category of based spaces. Let $BJ$ denote the classifying space of $J$.
The inclusions of the basepoints define a map
$$
BJ\to \hocolim D
$$
and the based homotopy colimit is the quotient $(\hocolim D)/BJ$. 

\begin{coro}\label{4_3}
Let $D: J\to\Top^\ast$ be a diagram of wellpointed Dold spaces and based maps. Then
based-$\hocolim D$ is a wellpointed Dold space.
\end{coro}

\begin{proof}
If $D: J\to\Top^\ast$ is a diagram of wellpointed spaces then 
$BJ\to \hocolim D$ is a closed cofibration. Apply \ref{3_6}.
\end{proof}

\begin{nr}\label{4_3a}\textbf{Remark:} 
The condition that $X$ be wellpointed can be achieved functorially by a 
\textit{whiskering process}: For a based space $(X,x_0)$ we define
$X_I=(X\sqcup I)/(x_0\sim 1)$ and choose $0\in I$ as basepoint of $X_I$. 
The natural map $q: X_I\to X$ mapping $I$ to $x_0$ is a homotopy equivalence.
If $X$ is h-wellpointed, it is even a based homotopy equivalence. We most
often state our results for wellpointed spaces because the pushout-product
theorem for cofibrations requires one factor to be a closed cofibration, but
for constructions which are homotopy invariant in the based category the
results extend to h-wellpointed spaces. An example  
is the following corollary:
\end{nr}

\begin{coro}\label{4_4} \cite[Lemma 1.5]{Pup1}
\begin{enumerate}
\item Given a diagram 
$$
\xymatrix{
X & \ar[l]_f A\ar[r]^g &Y 
}$$
of h-wellpointed spaces and based maps with $X$ and $Y$ Dold spaces. Then the reduced
mapping cylinder $M(f,g)$ is a Dold space.
\item Let $(X_{\alpha};\alpha\in A)$ be a family of h-wellpointed Dold spaces. Then
$\bigvee\limits_{\alpha\in A} X_\alpha$ is a Dold space.
\item Let $X$ and $Y$ be h-wellpointed Dold spaces. The $X\wedge Y$ is a Dold space.
\item The reduced suspension $SX$ of an h-wellpointed space is a Dold space.
\end{enumerate}
\end{coro}

\begin{rema}\label{4_5}
Example \ref{2_10} shows that \ref{4_4}.2 does not hold without some
assumptions on the basepoints.
\end{rema}

We next give a characterization of path-connected Dold spaces which needs some
preparations. Let
$$
p:E\to X
$$
be a map of based spaces and $F(p)$ the homotopy fiber of $p$ over the basepoint
$\ast$. With $p$ we associate a map
$$
q_\bullet: E_\bullet(p) \to \Omega_\bullet X
$$
of simplicial spaces as follows: $\Omega_n X\cong
\Top((\bigtriangleup^n,\bigtriangleup^n_0),(X,\ast))$ with the function space
topology. Here $\bigtriangleup^n_0$ is the $0$-skeleton of $\bigtriangleup^n$.
Boundaries and degeneracies are defined as for the singular functor. In fact,
$\Omega_\bullet (\;)$ is a topologized version of the singular functor
$$
{\Top}^\ast\to {r\Top}^{\bigtriangleup^{op}}
$$
from ${\Top}^\ast$ into the category of reduced simplicial spaces, 
i.e. simplicial spaces $Y_\bullet$
with $Y_0$ a point. It is right adjoint to the realization functor
${r\Top}^{\bigtriangleup^{op}} \to\Top^\ast$.

Let $C\bigtriangleup^n$ denote the cone on $\bigtriangleup^n$ with cone point $c_0$.
We define 
$$
E_n(p)=\{(e,w)\in E\times\Top((C\bigtriangleup^n,\bigtriangleup^n_0),(X,\ast)),
w(c_0)=p(e)\}.
$$
Boundaries and degeneracies are again defined by the corresponding maps of the
standard simplices. Finally, we define
$$
q_n:E_n(p)\to \Omega_nX, \quad (e,w)\mapsto w|\bigtriangleup^n
$$
Let $L_n\subset\bigtriangleup^n$ be the union of edges joining the $i$-th with the
$(i+1)$-st vertex of $\bigtriangleup^n$. Since $L_n\subset\bigtriangleup^n$ is a
strong deformation retract and the inclusion is a cofibration there is a fibration
and homotopy equivalence
$$
\Omega_nX \to \Top((L_n,L_n\cap \bigtriangleup^n_0), (X,\ast))\cong(\Omega X)^n.
$$
Since $C\bigtriangleup^n\cong\bigtriangleup^{n+1}$ we have a similar homotopy
equivalence 
$$
E_n(p) \to F(p)\times(\Omega X)^n,
$$
and
$$
\xymatrix{
E_n(p) \ar[rr] \ar[d]_{q_n} && 
F(p)\times(\Omega X)^n\ar[d]^{\textrm{proj.}}
\\
\Omega_nX \ar[rr] &&(\Omega X)^n
}
$$
commutes. 
Keeping this in mind, it is easy to show that $q_\bullet:
E_\bullet(p)\to\Omega_\bullet X$ is a simplicial object in the category Pull, whose
objects are maps and whose maps are commutative squares which are homotopy pullbacks.
A result of V.~Puppe \cite{VP} implies

\begin{nr}\label{4_6}
$$\xymatrix{
F(p)\ar[rr] \ar[d] && ||E_\bullet(p)|| \ar[d]^{||q_\bullet||}
\\
\ast \ar[rr] && ||\Omega_\bullet X||
}$$
is a homotopy pullback. The horizontal maps are the inclusions of the $0$-skeleta.
\end{nr}

Let $P(p)=\{(e,\alpha)\in E\times \Top(I,X); \alpha(0)=p(e)\}$ be the mapping path-space
of $p$. The maps
$$
\begin{array}{rcl}
E_n(p)\times\bigtriangleup^n & \longrightarrow & P(p)\\
(e,w,t) &\longmapsto &(e,\overline{w})
\end{array}
$$
with $\overline{w}(s)=w(s,t)$ for $(s,t)\in
C\bigtriangleup^n=(I\times\bigtriangleup^n)/(0\times\bigtriangleup^n)$, define a map 
$$
u:||E_\bullet (p)|| \longrightarrow  P(p).
$$
The counit $v:||\Omega_\bullet X||\to X$ of the adjoint pair

\begin{nr}\label{4_6a}
$$
\xymatrix{
||\ ?\ ||:{r\Top}^{(\Mon\bigtriangleup)^{op}}  \ar@<0.5ex>[r] & \ar@<0.5ex>[l]
{\Top}^\ast:\Omega_\bullet
} 
$$
\end{nr}
is induced by maps
$$
\Omega_nX\times \bigtriangleup^n\to X, \quad (\sigma,t)\mapsto \sigma(t),
$$
and we obtain a map of fiber sequences

\begin{nr}\label{4_7}
$\xymatrix{
F(p) \ar[rr]\ar[d]^\id && 
||E_\bullet(p)|| \ar[rr]^{||q_\bullet||}\ar[d]^u \ar @{} [drr] |{\textrm{I}} &&
||\Omega_\bullet X|| \ar[d]^v
\\
F(p) \ar[rr] && 
P(p) \ar[rr] && 
X
}$

Since $||\Omega_\bullet X||$ is a Dold space by \ref{4_1}, the square I is a
homotopy pullback by~\ref{1_6}.

Consider the case where $E$ is a point. Then $P(p)$ is contractible, and so is
$||E_\bullet (p)||$: Note that $E_n(p)\cong \Omega_{n+1}X$, so that $E_\bullet(p)$
is the based path-space construction $P\Omega_\bullet X$ in ${\Top}^{\bigtriangleup^{op}}$. It is
well-known that $||P\Omega_\bullet X||\simeq \Omega_0X=\ast$. So $u$ is a homotopy
equivalence. Hence $v$ is a homotopy equivalence by \ref{1_4}, provided $X$ is a
path-connected Dold space. We obtain
\end{nr}

\begin{prop}\label{4_8}
A path-connected space $X$ is a Dold space iff the counit $||\Omega_\bullet X||\to
X$ of the adjoint pair
 \ref{4_6a} is a homotopy equivalence.
\hfill $\square$
\end{prop}

Now let $E$ be any based space. Since $v$ is a homotopy equivalence provided $X$ is
a path-connected Dold space, $u$ is a homotopy equivalence.

\begin{prop}\label{4_9}
If $X$ is a path-connected Dold space, then for any based map $p:E\to B$ the maps 
$u$ and
$v$ of \ref{4_7} are homotopy equivalences.
\end{prop}

\begin{rema}\label{4_10}
From \ref{4_9} we obtain an alternative proof of Proposition \ref{3_9}. Let $p:E\to
X$ be a map, $X$ a path-connected Dold space, and suppose the homotopy fiber $F(p)$
is also a Dold space, then $E$ is a Dold space: Consider
$$
||E_\bullet(p)||\stackrel{v}{\to} P(p) \stackrel{r}{\to}E
$$
with $r(e,\alpha)=e$. The maps $v$ and $r$ are homotopy equivalences. Since
$E_0(p)=F(p)$ the space $||E_\bullet(p)||$ is a Dold space by \ref{4_1}, and hence
so is $E$.
\end{rema}

\begin{prop}\label{4_11}
Let $p_\ast:E_\ast\to X_\ast$ be a map of based semisimplicial spaces. Let $F(p_n)$
denote the homotopy fiber of $p_n:E_n\to X_n$. If each $X_n$ is a path-connected Dold
space, then
$$
\xymatrix{
||F(p_\ast)|| \ar[rr] \ar[d] && ||E_\ast|| \ar[d]
\\
\ast  \ar[rr] && ||X_\ast||
}
$$
is a homotopy pullback.
\end{prop}

\begin{proof}
By the naturality of our constructions we have a commutative diagram of semisimplicial
spaces
$$
\xymatrix{
F(p_\ast) \ar[rr] \ar[d]^{\id} &&
||E_\bullet(p_\ast)|| \ar[rr] \ar[d]^{u_\ast} &&
||\Omega_\bullet X_\ast|| \ar[d]^{v_\ast}
\\
F(p_\ast) \ar@{-}[rr] \ar[d]^{\id} &&
P(p_\ast) \ar[rr] \ar[d]^{r_\ast} &&
X_\ast \ar[d]^{\id}
\\
F(p_\ast) \ar[rr] &&
E_\ast \ar[rr]^{p_\ast} && X_\ast
}
$$
Since the vertical maps are homotopy equivalences in each degree they induce
homotopy equivalences of fat realizations. Hence it suffices to show that
$$
\xymatrix{
||[k] \mapsto F(p_k)|| \ar[rr] \ar[d] &&
||[k] \mapsto ||E_\bullet(p_k)||\;|| \ar@<-3ex>[d]
\\
\ast \ar[rr] &&
||[k] \mapsto ||\Omega_\bullet X_k||\;||
}
$$
is a homotopy pullback. For this we study the map
$$
q_{n,k}:E_n(p_k) \to \Omega_n(p_k)
$$
of bi-semisimplicial spaces. Since its total fat realization is independent of the
order in which we realize, we may first realize with respect to $k$ and obtain a map of
semisimplicial spaces
$$
\overline{q}_n:||E_n(p_\ast)|| \to ||\Omega_nX_\ast||
$$
%\end{proof}

\textbf{Claim:} $\overline{q}_n$ is a semisimplicial object in Pull, i.e.
$$
\xymatrix{
||E_n(p_\ast)|| \ar[rr]^{d^i}\ar[d]_{\overline{q}_n} &&
||E_{n-1}(p_\ast)|| \ar[d]^{\overline{q}_{n-1}} 
\\
||\Omega_nX_\ast|| \ar[rr]^{d^i} && 
||\Omega_{n-1}X_\ast||
}
$$
is a homotopy pullback for each $n$.

\textbf{Proof:}
Let $j\neq i$. There is a strong deformation retraction of $C\bigtriangleup^n$ to
$\bigtriangleup^n\cup_{v_j}[c_0,v_j]$, where $[c_0,v_j]$ is the line from the cone
point $c_0$ to the $j$-th vertex $v_j$ of $\bigtriangleup^n$. This
deformation retraction can be chosen
compatibly with $d^i$ yielding a commutative square
$$
\xymatrix{
E_n(p_k) \ar[rr]\ar[d]_{d^i} &&
F(p_k)\times \Omega_n X_k \ar[d]^{\id \times d^i}
\\
E_{n-1}(p_k) \ar[rr] && F(p_k)\times \Omega_{n-1} X_k
}
$$
whose horizontal maps are homotopy equivalences. Since the fat realization
preserves products up to homotopy, we obtain a commutative diagram
$$
\xymatrix{
||E_n(p_\ast)|| \ar[rr]\ar[d]_{d^i} &&
||F(p_\ast)||\times ||\Omega_nX_\ast|| \ar[d]^{\id\times d^i}
\\
||E_{n-1}(p_\ast)|| \ar[rr] &&
||F(p_\ast)||\times ||\Omega_{n-1}X_\ast||
}
$$
whose horizontal maps are homotopy equivalences. So it suffices to show that
$$
\xymatrix{
||F(p_\ast)||\times ||\Omega_nX_\ast||
\ar[rr]^{\id\times d^i}\ar[d]^{\textrm{proj.}}&&
||F(p_\ast)||\times ||\Omega_nX_\ast|| \ar[d]^{\textrm{proj.}}
\\
||\Omega_nX_\ast||  \ar[rr]^{d^i} &&
||\Omega_nX_\ast||
}
$$
is a homotopy pullback. But this is evident. This proves the claim.

We now apply V. Puppe's result \cite{VP} again: Since $||F(p_\ast)||$ is  the
$0$-skeleton of $||[n]\mapsto ||E_n(p_\ast)||\;||$, we obtain a homotopy
pullback
$$
\xymatrix{
||F(p_\ast)|| \ar[rr]\ar[d] &&
||E_\bullet(p_\ast)|| \ar[d]
\\
\ast \ar[rr] &&
||\Omega_\bullet X_\ast||
}
$$
\end{proof}

The basic idea of the argument of the previous proof is due to D. Puppe. He used it
so show that the fat realization commutes with the loop space functor for
path-connected semisimplicial Dold spaces.

\begin{prop}\label{4_12}
Let $X_\bullet$ be a based semisimplicial space such that each $X_n$ is a
path-connected Dold space. Then there is a canonical homotopy equivalence
$$
||\Omega X_\bullet|| \to \Omega||X_\bullet||
$$
In particular, $\Omega||X_\bullet||$ is a Dold space if $\Omega X_0$ is a Dold space
(e.g. if $X_0$ is based contractible).
\end{prop}

\begin{proof}
Let $\ast$ denote the semisimplicial point. Apply \ref{4_11} to the map $p_\bullet:
\ast\to X_\bullet$. Since $F(p_\bullet)=\Omega X_\bullet$ and $||\ast||$ is contractible,
the statement follows.

\end{proof}

We will see that loop spaces of Dold spaces need not be Dold spaces. But we have

\begin{prop}\label{4_13}
If $X$ is a path-connected h-wellpointed Dold space, then $\Omega\Sigma X$ is a Dold
space.
\end{prop}

\begin{proof}
By Remark \ref{4_3a} we may assume that $X$ is wellpointed.
Let $X_\bullet$ be the simplicial space which is the $n$-fold wedge of $X$ in degree
$n$. Boundaries $d^i$ are the folding map for $0<i<n$ and projections for $i=0,n$.
Degeneracies are the obvious injections. Apply \ref{4_11} to $p_\bullet:\ast\to
X_\bullet$. Since $X_\bullet$ is proper, we have homotopy equivalences
$$
||F(p_\bullet)|| = ||\Omega X_\bullet||
\simeq \Omega ||X_\bullet||\simeq \Omega |X_\bullet|=\Omega\Sigma X.
$$
Since $F(p_0)$ is a point, $||F(p_\bullet)||$ is a Dold space.
\end{proof}

In his proof of \ref{4_12} Puppe used the nerve $N_\bullet\Omega_M X$ of the Moore loop
space $\Omega_MX$ of $X$ rather than $\Omega_\bullet X$. In fact, there is a
simplicial map
$$
\alpha_\bullet:N_\bullet\Omega_M X\to\Omega_\bullet X
$$
which is degreewise a homotopy equivalence inducing a homotopy equivalence
\cite[Appendix]{Brink}
$$
||\alpha_\bullet||: ||N_\bullet\Omega _MX||\to || \Omega_\bullet X||.
$$
We note that $N_\bullet\Omega_M X$ is proper if $X$ is wellpointed, because
$\Omega_MX$ is wellpointed respectively h-wellpointed
if $X$ is \cite[(11.3)]{DKP}. We obtain

\begin{coro}\label{4_14}
Suppose $X$ is a wellpointed path-connected space. Then $X$ is a Dold space iff
$v\circ|\alpha_\ast|:\mathcal{B}\Omega_MX\to X$ is a homotopy equivalence, where $B$
is the classifying space functor and $v$ is the map of \ref{4_8}.
\hfill$\square$
\end{coro}

The following result is an extension of Proposition \ref{4_11}.

\begin{prop}\label{4_15}
Given a commutative diagram of based simplicial spaces
$$
\xymatrix{
A_\ast \ar[rr]^{f_\ast} \ar[d]_{q_\ast} &&
E_\ast \ar[d]^{p_\ast}
\\
B_\ast \ar[rr]^{g_\ast} && X_\ast
}
$$
which is a homotopy pullback in each dimension. If each $B_n$ and each $X_n$ is a
path-connected Dold space, then
$$
\xymatrix{
||A_\ast|| \ar[rr]^{||f_\ast||} \ar[d]^{||q_\ast|} &&
||E_\ast|| \ar[d]^{||p_\ast||}
\\
||B_\ast|| \ar[rr]^{||g_\ast||} &&
||X_\ast||
}
$$
is a homotopy pullback.
\end{prop}

\begin{proof}
From \ref{4_11} we obtain a diagram
$$
\xymatrix@R=2ex@C=2ex{
||F(q_\ast)|| \ar[rrr] \ar[dd]  \ar[rd] && &
||A_\ast|| \ar[rd] \ar[dd] |!{[d];[d]}\hole
\\
& ||F(p_\ast)|| \ar[rrr] \ar[dd] &&&
||E_\ast|| \ar[dd]
\\
\ast \ar[rd] \ar[rrr] |!{[rr];[r]}\hole &&&
||B_\ast|| \ar[rd]
\\
& \ast \ar[rrr] &&& 
||X_\ast||
}
$$
whose front face $(F)$ and back face $(B)$ are homotopy pullbacks. Since the map
$F(q_\ast)\to F(p_\ast)$ is a homotopy equivalence in each dimension by assumption,
its realization is a homotopy equivalence. Hence the left face $(L)$ is a homotopy
pullback. If $(R)$ denotes the right face, we find that $(B)+(R)$ is a homotopy
pullback, because $(L)+(F)$ is one. Hence $(R)$ is a homotopy pullback by \ref{1_5}.
\end{proof}

%sec5
\section{$\mathbf{k}$-spaces and free algebras over operads}
Throughout this section we work in the category $k$Top of $k$-spaces and its based
version $k\Top^\ast$. Recall that $X$ is a $k$-space if a subset $U\subset X$ is
open precisely if $f^{-1}(U)$ is open for all maps $f:C\to X$ and all compact
Hausdorff spaces $C$. The inclusion functor $i:k\Top\subset\Top$ has a right adjoint
$k:\Top\to k\Top$ obtained from $X$ by declaring the subsets $U$ satisfying the
condition above as open. The counit of this adjunction
$$
ik(X)\to X
$$
is the identity on underlying sets. 
Hence the topology of $k(X)$ is finer than the one of $X$, and we
obtain

\begin{prop}\label{5_1}
If $X$ is a Dold space so is $k(X)$.
\end{prop}

Since $i$ preserves colimits and $k$ limits, we moreover have

\begin{prop}\label{5_2}
The results of the previous sections also hold in the category $k\Top$ respectively
$k\Top^\ast$.
\end{prop}

We include $k\Top$ into our considerations because Theorems \ref{1_1} and \ref{1_2}
are phrased in $k\Top^\ast$.

In his proof of Theorem \ref{1_2} Meiwes needed to show that $\mathbb{C}^\ast_n(X)$ and the
$k$-fold symmetric product $SP_k(X)$ are Dold spaces if $X$ is a (wellpointed for
$\mathbb{C}^\ast_n(X)$) Dold space. He did this by explicitely constructing Dold covers. We will
obtain these results from more general easy to prove facts.

Let $\mathcal{P}$ be a topological operad. We call $\mathcal{P}$ reduced if
$\mathcal{P}(0)$ consists of a single element. If $X$ is $\mathcal{P}$-space and
$\mathcal{P}$ is reduced, the single element of $\mathcal{P}(0)$ determines a basepoint
in $X$. Let $\mathcal{P}\Top$ be the category of $\mathcal{P}$-spaces. We have
a forgetful functor
$$
U:\mathcal{P}\Top \to k\Top
$$
and, if $\mathcal{P}$ is reduced,
$$
U^\ast: \mathcal{P}\Top \to {k\Top}^\ast.
$$
They have left adjoints
$$
\mathbb{P}:k\Top\to \mathcal{P}\Top \quad\textrm{ respectively }\quad
\mathbb{P}^\ast:{k\Top}^\ast\to\mathcal{P}\Top
$$
defined by
$$
\mathbb{P}(X)=\coprod\limits^\infty_{n=0}\mathcal{P}(n)\times_{\Sigma_n}X^n
\quad \textrm{ and }\quad
\mathbb{P}^\ast(X)\left(\coprod\limits^\infty_{n=0}\mathcal{P}(n)\times_{\Sigma_n}
X^n\right) /\sim
$$

The relation $\sim$ in the definition of $\mathbb{P}^\ast(X)$ is defined as follow:
Let $\ast\in\mathcal{P}(0)$ denote the single element and let
$$
\begin{array}{ll}
\sigma_i:\mathcal{P}(k) \to \mathcal{P}(k-1), &
\alpha\mapsto \alpha\circ(\id_i\times\ast\times \id_{k-i-1})
\\
s_i: X^{k-1} \to X^k, &
(x_1,\ldots,x_{k-1})\mapsto (x_1,\ldots,x_i,\ast,x_{i+1},\ldots,x_{k-1})
\end{array}
$$
Then $(\sigma_i(\alpha),x)\sim (\alpha,s_i(x))$.

\begin{prop}\label{5_3}
Let $\mathcal{P}$ be an operad such that each $\mathcal{P}(n)/\Sigma_n$ is a Dold
space, and let $X\in k\Top$ be a path-connected Dold space.
Then $\mathbb{P}(X)$ is a Dold space.
\end{prop}

\begin{prop}\label{5_4}
Let $X\in k\Top^\ast$ be a wellpointed path-connected Dold sapce. Then
$\mathbb{P}^\ast(X)$ is a Dold space for each reduced operad $\mathcal{P}$.
\end{prop}

The proofs make use of the following result of May \cite[Thm. 12.2]{May}:

\begin{prop}\label{5_5}
\begin{enumerate}
\item Let $X_\bullet$ be a simplicial $k$-space, then there is a natural
homeomorphism $|\mathbb{P}(X_\bullet)|\to\mathbb{P}(|X_\bullet|)$.
\item Let $X_\bullet$ be a wellpointed simplicial $k$-space and $\mathcal{P}$ a
reduced operad. Then there is a natural homemorphism
$|\mathbb{P}^\ast(X_\bullet)|\to\mathbb{P}^\ast(|X_\bullet|)$.
\end{enumerate}
\end{prop}

May proves the based case, but the proof applies verbatim also to the non-based case.

\textbf{Proof of \ref{5_3}:} Choose a basepoint $x_0\in X$ and let
% $1\in I=[0,1]$. Let
%$X_I=X\vee I$, but with $0\in I$ as basepoint, and let $q:X_I\to X$ be the map
%shrinking $I$ to a point. Then $q$ is a homotopy equivalence. Since $X_I$ is a
$q:X_I\to X$ be the map of \ref{4_3a}. Since $X_I$ is a
wellpointed Dold space the map
$$
v\circ|\alpha_\ast|:|N_\ast\Omega_MX_I|\to X_I
$$
of \ref{4_14} is a based
homotopy equivalence. By \ref{5_5} we have a sequence of homotopy
equivalences (we ignore basepoints)
$$
|\mathbb{P}(N_\ast\Omega_MX_I)|\to\mathbb{P}(|N_\ast\Omega_MX_I|)\to\mathbb{P}(X_I)
\to \mathbb{P}(X).
$$
$N_0\Omega_MX_I$ is a single point. Hence
$$
\mathbb{P}(N_0\Omega_MX_I)=\mathbb{P}(\ast)=\coprod\limits^\infty_{n=0}\mathbb{P}(n)
/\Sigma_n
$$
which is a Dold space. Hence $|\mathbb{P}(N_\ast\Omega_MX_I)|$ and $\mathbb{P}(X)$
are Dold spaces by \ref{4_1}, because $\mathbb{P}(N_\ast\Omega_MX_I)$ is proper.
\hfill$\square$
\vspace{2ex}

\textbf{Proof of \ref{5_4}:}
If $X$ is wellpointed the map $q:X_I\to X$ is a based homotopy equivalence. We
obtain a sequence of based homotopy equivalences
$$
|\mathbb{P}^\ast(N_\ast\Omega_MX_I)|\to \mathbb{P}^\ast(|N_\ast\Omega_MX_I|) \to
\mathbb{P}^\ast(X_I)\to \mathbb{P}^\ast(X).
$$
Since $\mathbb{P}^\ast(\ast)=\ast$, all spaces are Dold spaces by \ref{4_1}, because
$
\mathbb{P}^\ast(N_\ast\Omega_MX_I)$ is proper.
\hfill$\square$

\begin{coro}\label{5_6}
Let $X\in k\Top$ be a Dold space. Then $SP_k(X)$ is a Dold space.
\end{coro}

\begin{proof}
Let $\{X_\alpha; \alpha\in A\}$ be the set of path-components of $X$. Then $SP_k(X)$
is the disjoint union of spaces
$$
SP_{r_1}(X_{\alpha_1})\times\ldots\times SP_{r_q}(X_{\alpha_q}), \qquad
r_1+\ldots+r_q=k.
$$
By \ref{3_1} and \ref{3_10} it suffices to prove the result for path-connected $X$.

Let $\mathcal{C}om$ be the operad for commutative monoid structures, i.e.
$\mathcal{C}om(n)$ is a single point for each $n$. Then
$$
\mathcal{C}om(X)=\coprod\limits_n SP_n(X)
$$
is a Dold space by \ref{5_3}. So $SP_n(X)$ is a Dold space by \ref{3_1}.
\end{proof}

%\sec6
\section{Counter examples}
\begin{prop}\label{6_1}
There are Dold spaces which are not of the homotopy type of a $CW$-complex.
\end{prop}

The following example was brought to our attention by A. Hatcher \cite{Hat}.

Let $Y=\{\frac{1}{n};\ n\in\mathbb{N}\}\cup\{0\}\subset\mathbb{R}$ and let
$\widehat{\Sigma} Y$ be the
unreduced suspension of $Y$. Then $\widehat{\Sigma} Y$ is a Dold space. Let 
$f:\widehat{\Sigma} Y\to X$
be any map into a $CW$-complex. Since $\widehat{\Sigma} Y$ is compact, 
$f$ factors through a
finite subcomplex $A\subset X$. Hence $f_\ast: H_1(\widehat{\Sigma} Y)\to H_1(X)$
factors through $H_1(A)$. Let $B_n$ be the unreduced suspension of
$\{0\}\cup \{\frac{1}{i};\ 1\leq i\leq n\}$. Then $B_n$ is a retract of
$\widehat{\Sigma} Y$ and $H_i(B_n)\cong \mathbb{Z}^n$. Hence 
$H_1(\widehat{\Sigma} Y)$ is not 
finitely generated, but $H_1A$
is. So the map $f$ cannot be a homotopy equivalence.

\begin{coro}\label{6_2}
There are weak homotopy equivalences between Dold spaces which are not homotopy
equivalences.
\end{coro}

\textbf{Example:}
Let $g:R\widehat{\Sigma} Y\to \widehat{\Sigma} Y$ be a $CW$-approximation  of 
$\widehat{\Sigma} Y$ of the previous
example. Then $g$ is a weak equivalence but not a homotopy equivalence.

\begin{prop}\label{6_3}
The loop space of a Dold space need not be a loop space.
\end{prop}

\textbf{Example:}
$\mathbb{Q}$ with the subspace topology of $\mathbb{R}$
is not a Dold space. Let $N_\bullet\mathbb{Q}$ denote the nerve of
$(\mathbb{Q},+)$. Since $(\mathbb{Q},+)$ is a 
topological group there is a homotopy equivalence
$\mathbb{Q}\to\Omega||N_\bullet\mathbb{Q}||$. Hence 
$\Omega||N_\bullet\mathbb{Q}||$ is not a Dold space but 
$||N_\bullet\mathbb{Q}||$ is one.
\hfill$\square$

\vspace{2ex}
Recall that a \textit{closed class} $\mathcal{C}$ in the sense of Dror Farjoun is a
full subcategory of the category $S_\ast$ of wellpointed spaces of the homotopy
type of a $CW$-complex which is closed under homotopy equivalences and based
homotopy colimits \cite[D1]{Dror}.

The class of wellpointed Dold spaces is closed under homotopy equivalences and
based homotopy  colimits, but Dror Farjoun's results do not generalize to this
class.

\begin{exa}\label{6_4}
Let $F\to E\to B$ be a fibration sequence with path-connected $B$. If $F$ and $E$ are
in a closed class $\mathcal{C}$, then so is $B$ \cite[D. 11]{Dror}.

This does not hold for Dold spaces. Consider the based path-space fibration
$$
\Omega C\to PC\to C
$$
over the polish circle $C$ with a nice point $c_0\in C$. It is well-known that
$\Omega C$ and $PC$ are contractible and hence Dold spaces, but $C$ is not a Dold
space.
\end{exa}

\end{document}